%% file: ACT25_12pp.tex
\documentclass[submission,copyright,creativecommons]{eptcs}
\providecommand{\event}{ACT 2025} 

\usepackage{iftex}

\ifpdf
  \usepackage{underscore}         
  \usepackage[T1]{fontenc}        
\else
  \usepackage{breakurl}           
\fi

\input{preamble}

\title{A Critical Pair Enumeration Algorithm for String Diagram Rewriting\footnote{%
We are grateful to organisers of Adjoint School 2024 in which our collaboration started.}}
\author{%
Anna Matsui
\institute{Johns Hopkins University, USA}
\email{amatsui1@jhu.edu}
\and
Innocent Obi
\institute{University of Washington, USA}
\email{innoobi@cs.washington.edu}
\and
Guillaume Sabbagh
\institute{University of Technology of Compiègne, France\thanks{G.S.\ was funded by BNP Paribas CIB EMEA and the French Ministry of Research under CIFRE project No. 2021/1502.}}
\email{guillaume.sabbagh@utc.fr}
\and
Leo Torres
\institute{Universidad Nacional de Córdoba, Argentina}
\email{leo.torres@mi.unc.edu.ar}
\and
Diana Kessler
\institute{Tallinn University of Technology, Estonia\thanks{D.K.\ was funded by an Advanced Research + Invention Agency (ARIA) Safeguarded AI: TA1.1 Theory grant.}}
\email{diana-maria.kessler@taltech.ee}
\and
Juan F. Meleiro
\institute{University of São Paulo, Brazil}
\email{juan.meleiro@usp.br}
\and
Koko Muroya
\institute{National Institute of Informatics, Japan \& Ochanomizu University, Japan\thanks{K.M.\ was funded by JSPS, KAKENHI Project No.\ 22K17850, Japan.}}
\email{kmuroya@is.ocha.ac.jp}
}

\def\titlerunning{A Critical Pair Enumeration Algorithm for String Diagram Rewriting}
\def\authorrunning{A. Matsui et al.}

\begin{document}

\maketitle              
\begin{abstract}
 Critical pair analysis provides a convenient and computable criterion of confluence, which is a fundamental property in rewriting theory, for a wide variety of rewriting systems. Bonchi et al. showed validity of critical pair analysis for rewriting on string diagrams in symmetric monoidal categories. This work aims at automation of critical pair analysis for string diagram rewriting, and develops an algorithm that implements the core part of critical pair analysis. The algorithm enumerates all critical pairs of a given left-connected string diagram rewriting system, and it can be realised by concrete manipulation of hypergraphs. We prove correctness and exhaustiveness of the algorithm, for string diagrams in symmetric monoidal categories without a Frobenius structure.

\end{abstract}
\section{Introduction}

\subsection{Rewriting Theory and Critical Pair Analysis}

Mathematical reasoning often involves derivation of a (complex) equation from known (typically simpler) equations, which is sometimes called \emph{equational reasoning}. Equations can be between various mathematical objects, e.g. terms, programs, graphs, processes, and objects/morphisms in a category.

Rewriting theory has been established (see e.g.\ \cite{TRStext,GraphTransfI}), with equational reasoning as one application.
The starting point is to turn known equations $a = a'$ into directed\footnote{The direction is typically chosen so that $a'$ is ``simpler'' than $a$.} \emph{rewrite rules} $a \multimap a'$. Each step $b \to b'$ of rewrite modifies a part of $b$ by applying one rewrite rule. Derivation of an equation $c \overset{?}{=} c'$ then boils down to finding some $d$ with two chains $c \rightarrow \cdots \rightarrow d$ and $c' \rightarrow \cdots \rightarrow d$ of rewrites. These chains altogether imply a chain of equations $c = \cdots = d = \cdots = c'$, which concludes the desired equation $c \overset{?}{=} c'$.

\emph{Confluence} is a fundamental property in rewriting theory, intuitively meaning that ordering of rewrites does not matter. Rewrites $\rightarrow$ are said to be confluent if any two diverging chains $b \leftarrow \cdots \leftarrow a \rightarrow \cdots \rightarrow b'$ of rewrites are \emph{joinable}, that is, there exists $c$ with converging chains $b \rightarrow \cdots \rightarrow c \leftarrow \cdots \leftarrow b'$.
\emph{Local confluence} is a variant of confluence in which the two diverging chains are in fact given by two single rewrites, i.e.\ $b \leftarrow a \rightarrow b'$.

For some pairs of two diverging rewrites $b \leftarrow a \rightarrow b'$, joinability is obvious. An example is the so-called \emph{parallel} case, namely when the two rewrites change different, independent, parts of $a$. Consequently, checking local confluence boils down to analysing joinability of non-parallel pairs.

\emph{Critical pair analysis} is a well-established technique for automatically checking local confluence, providing a convenient and computable criterion. It reduces local confluence to joinability of \emph{critical pairs} that are finitely many representatives of non-parallel pairs. Critical pairs can be enumerated from a given set of rewrite rules $a \multimap a'$. This enumeration plays a central role in automating local-confluence check.

\subsection{String Diagram Rewriting}

\emph{String diagrams} \cite{JoyalS91,Selinger10} provide a graphical syntax of category theory. They are useful in equational reasoning on morphisms of a category, because they trivialise certain equations as graph isomorphisms.

Rewriting theory for string diagrams has been developed by Bonchi et al.\ \cite{DBLP:journals/jacm/BonchiGKSZ22,DBLP:journals/mscs/BonchiGKSZ22,DBLP:journals/mscs/BonchiGKSZ22a}, targeting at string diagrams for symmetric monoidal categories (with and without a Frobenius structure).
String diagrams are combinatorially represented using hypergraphs, and rewrites on string diagrams are categorically modelled using \emph{double pushout rewriting} (DPO rewriting in short) \cite{ehrigDPO1973}. A key concept in string diagram rewriting theory is that of \emph{interface}. An interface of a hypergraph specifies how other hypergraphs can be connected to the hypergraph.

Bonchi et al.\ showed validity of critical pair analysis for string diagram rewriting \cite{DBLP:journals/mscs/BonchiGKSZ22a}. They defined critical pairs for an adaptation of DPO rewriting (dubbed \emph{convex\footnote{Convexity is for dealing with the absence of a Frobenius structure.} DPOI rewriting}) that takes interface into account, and proved that joinability of critical pairs implies local confluence.
Their development focuses on a theoretical side, and automation, which is an important aspect of critical pair analysis, has not been investigated.

\subsection{Contributions}

We aim at automation of critical pair analysis for string diagram rewriting, and develop an algorithm that implements the core part of the automation. The algorithm enumerates all critical pairs for a given set of DPOI rewrite rules.
We focus on the so-called \emph{left-connected} DPOI rewrite rules \cite{DBLP:journals/mscs/BonchiGKSZ22,DBLP:journals/mscs/BonchiGKSZ22a}. Left-connectivity allows us to reduce enumeration of critical pairs to enumeration of certain cospans in the category of hypergraphs. While it is an arguably powerful restriction, it still accommodates various concrete string diagram rewriting systems from the literature \cite{DBLP:conf/birthday/FioreC13,DBLP:conf/birthday/Ghica13,LAFONT2003257}.

Each critical pair is associated with two DPOI rewrite rules, which are given by spans $L_1 \leftarrow K_1 \rightarrow R_1$ and $L_2 \leftarrow K_2 \rightarrow R_2$ in the category of hypergraphs. Thanks to left-connectivity, the critical pair is uniquely determined by a certain cospan of the form $L_1 + L_2 \twoheadrightarrow S \leftarrow J$. Its left leg is, in particular, an epimorphism given by the coupling of monomorphisms.

Our key idea is that the cospan, in particular the hypergraph $S$, can be generated by suitably \emph{gluing} hyperedges and nodes of $L_1 + L_2$ (i.e.\ the hypergraph that puts $L_1$ and $L_2$ in parallel). We observe that the gluing process can be realised in two steps: (1) repeatedly merge a hyperedge from $L_1$ with a hyperedge from $L_2$, and (2) repeatedly merge a node from $L_1$ with a node from $L_2$ without merging any hyperedges.

Our contributions can be summarised as follows.
\begin{itemize}
 \item We develop an algorithm (Algo.~\ref{algo3}) that enumerates all critical pairs of a given set of left-connected DPOI rewrite rules by implementing the two-fold gluing process.
 \item We prove that the algorithm generates all critical pairs and nothing else (correctness and exhaustiveness; Thm.~\ref{thm:main}).
 \item We provide a proof-of-concept Haskell implementation\footnote{available online at \url{https://github.com/GuiSab/hypergraphrewriting}}.
 \item We present an optimised algorithm (Algo.~\ref{algo4}) that enumerates less but sufficient critical pairs to decide local confluence by only performing the first step of the two-fold gluing process.
\end{itemize}

\paragraph{Organisation.}
Sec.~\ref{sec:definitions} recalls relevant concepts (e.g.\ hypergraph, interface, DPOI rewriting, critical pair) from string diagram rewriting theory \cite{DBLP:journals/jacm/BonchiGKSZ22,DBLP:journals/mscs/BonchiGKSZ22,DBLP:journals/mscs/BonchiGKSZ22a}. Sec.~\ref{sec:algo} presents our main contribution, the critical pair enumeration algorithm with a proof of its correctness and exhaustiveness. Sec.~\ref{sec:opt} provides the optimised algorithm, and Sec.~\ref{sec:conclusion} concludes the paper.
Examples and some proofs can be found in the Appendix.

\paragraph{Related work.}
For term rewriting, rewrite rules typically use variables as placeholders to succinctly represent a family of rewrite rules, e.g.\ $x + y \multimap y + x$. To deal with variables, enumeration of critical pairs employs a technique called \emph{unification}.
In contrast, for string diagram rewriting, rewrite rules are always concrete without placeholders. We can therefore take a direct approach and generate a critical pair by suitably gluing hyperedges and nodes of left-hand sides of rewrite rules.
There are some attempts at enumerating critical pairs for variations of graph rewriting (graph transformation), e.g.\ \cite{DBLP:conf/rta/Mimram10,10.1007/978-3-319-40530-8_8,DBLP:phd/ethos/Hristakiev18}.

\section{Critical Pairs for String Diagram Rewriting} \label{sec:definitions}

We denote the composition of morphisms $f \colon A \to B$ and $g \colon B \to C$ by $f ; g$, and
coprojections of a coproduct by $\iota_1,\iota_2$.
Given a set $A$, the free monoid on $A$ is denoted by $A^*$.
Pointwise application a function $f \colon A \to B$ over a list of elements yields $f^* \colon A^* \to B^*$.
Let $\mathbb{N}$ be the set of natural numbers.

\subsection{Hypergraphs with Interface}

When a symmetric monoidal category is equipped with a Frobenius structure, string diagrams in the category can be combinatorially represented as (edge-labelled) hypergraphs with interface \cite{DBLP:journals/jacm/BonchiGKSZ22}.

\begin{definition}[Hypergraphs]
A (directed) \textbf{hypergraph} is a tuple $G = (V, E, s : E \rightarrow V^{*}, t : E \rightarrow V^{*})$ where $V$ and $E$ are finite sets of nodes and hyperedges, $s$ maps each hyperedge to a list of source nodes and $t$ maps each hyperedge to a list of target nodes. The \textbf{arity} of a hyperedge is the number of its sources, the \textbf{coarity} of a hyperedge is the number of its targets.

We refer to $V$ as $\Nodes(G)$ and $E$ as $\HEdges(G)$.
 
Let $\sigma$ be an alphabet, a \textbf{signature} $\Sigma$ on $\sigma$ is a subset of $\sigma \times \mathbb{N} \times \mathbb{N}$. A triplet $(x,n,m)$ represents a label $x$ for morphisms with arity $n$ and coarity $m$.
A $\Sigma$\textbf{-labelled hypergraph} ($\Sigma$-hypergraph in short) is a hypergraph equipped with a labelling function  $l : E \to \Sigma$ such that $l_E$ maps a hyperedge with arity $n$ and coarity $m$ to a triplet $(x,n,m)$.
\end{definition}

\begin{definition}[Hypergraph morphisms]
A \textbf{$\Sigma$-hypergraph morphism} between $\Sigma$-hypergraphs $(V_0,\allowbreak E_0,\allowbreak s_0,\allowbreak t_0,\allowbreak l_0)$ and $(V_1, E_1, s_1, t_1,l_1)$ is a pair of functions $f_V : V_0 \to V_1$ and $f_E : E_0 \to E_1$ that respects sources, targets and labels; that is, that satisfies
$f_V^* \circ s_0 = s_1 \circ f_E$, $f_V^* \circ t_0 = t_1 \circ f_E$ and $l_0 = l_1 \circ f_E$.
\end{definition}

Given a signature $\Sigma$, $\Sigma$-hypergraphs and $\Sigma$-hypergraph morphisms form a category $\Hyp_\Sigma$.
It has all small limits and colimits (since it is a presheaf category \cite[pp. 18-19]{DBLP:journals/jacm/BonchiGKSZ22}); in particular, it has pushouts, coproducts and coequalizers. We can spell them out in set-theoretic terms, which makes them suitable for an algorithmic implementation. 

Interface specifies nodes of a hypergraph to which other hypergraphs can be connected.
\begin{definition}[$\Sigma$-hypergraph with interfaces]
A discrete $\Sigma$-hypergraph is a $\Sigma$-hypergraph with no hyperedges (i.e. $G$ is discrete if $\HEdges(G)$ is the empty set, it only contains nodes). 
A $\Sigma$\textbf{-hypergraph with interface} is a cospan $\catcospan n G m$ in $\Hyp_{\Sigma}$ where $n$, $m$ are finite discrete $\Sigma$-hypergraphs. 
\end{definition}
The discrete hypergraphs $n$ and $m$ specify input interface and output interface, respectively.
We sometimes identify a $\Sigma$-hypergraph with interface $\catcospan n G m$ by a single morphism $G \leftarrow n + m$.

For a general symmetric monoidal category without a Frobenius structure, the combinatorial representation of string diagrams requires extra conditions on hypergraphs \cite{DBLP:journals/mscs/BonchiGKSZ22}: monogamy and acyclicity.

\begin{definition}[Paths]
    A \textbf{path} $P$ in a $\Sigma$-hypergraph is a list of hyperedges $[e_1,e_2,\cdots,e_n]$ such that for every consecutive pair of hyperedges $(e_k,e_{k+1})$, there is at least one target of $e_k$ equal to a source of $e_{k+1}$.

    A \textbf{cycle} $C$ in a $\Sigma$-hypergraph is a path such that at least one source of $e_1$ is a target of $e_n$.
\end{definition}

\begin{definition}[Monogamous acyclicity]
A $\Sigma$-hypergraph is \textbf{monogamous acyclic} (ma-hypergraph) if
\begin{enumerate}
    \item it contains no cycles (acyclicity) ;
    \item every node has at most in- and out- degree 1 (monogamy).
\end{enumerate}

Here in- (out-) degree of a node $v$ in a $\Sigma$-hypergraph $H$ is the number of pairs $(e, i)$ where $e$ is a hyperedge of $H$ with $v$ as its $i$-th target (source). We call input nodes those with in-degree 0, denoted by $in(H)$. Similarly, output nodes have out-degree 0 and are denoted by $out(H)$.

A $\Sigma$-hypergraph with interface $n \xrightarrow{f} H \xleftarrow{g} m$ is \textbf{monogamous acyclic}, or \textbf{ma-cospan}, if $H$ is an ma-hypergraph, $f$ is mono and its image is $in(H)$, and $g$ is mono and its image is $out(H)$.
\end{definition}

\subsection{Convex DPOI Rewriting}

Rewriting on string diagrams can be modelled categorically \cite{DBLP:journals/jacm/BonchiGKSZ22,DBLP:journals/mscs/BonchiGKSZ22}, by adapting DPO rewriting \cite{ehrigDPO1973}. We first recall DPO rewriting in $\Hyp_\Sigma$.

A rewrite rule is a span $\catspan L K R$ in $\Hyp_{\Sigma}$. A rewrite system $\mathcal{R}$ is a finite set of rewrite rules. We say that a $\Sigma$-hypergraph $G$ rewrites into a $\Sigma$-hypergraph $H$ if there exists a rule $L \xleftarrow{} K \to R$, a morphism $L \xrightarrow{m} G$ (called \emph{match}) and an object $C \in \Hyp_{\Sigma}$ such that the following two squares are pushouts: 

\vspace{-0.3cm}
\[\begin{tikzcd}[]
	{L} && {K} && {R} \\
	{G} && {C} && {H}
	\arrow["m"', from=1-1, to=2-1]
	\arrow[""{name=0, anchor=center, inner sep=0}, "f"', from=1-3, to=1-1]
	\arrow[""{name=1, anchor=center, inner sep=0}, "g", from=1-3, to=1-5]
	\arrow[from=1-3, to=2-3]
	\arrow[from=1-5, to=2-5]
	\arrow[from=2-3, to=2-1]
	\arrow[from=2-3, to=2-5]
	\arrow["\lrcorner"{anchor=center, pos=0.125, rotate=90}, draw=none, from=2-1, to=0]
	\arrow["\lrcorner"{anchor=center, pos=0.125, rotate=180}, draw=none, from=2-5, to=1]
\end{tikzcd}\]
\vspace{-0.3cm}

The above rewrite works as follows.
Computing the pushout complement removes the image of $L$ (the left-hand side of the rewrite rule) in $G$ while keeping the image of $K$ intact. 
By computing the pushout of $\catspan C K R [][g]$, we glue $R$ and $C$ along the image of $K$, thus replacing the image of $L$ in $G$ with $R$. 
More intuitively, what this procedure does is to take away the part that corresponds to the matching of the left-hand side, $L$, of a rewrite rule and replace it by its right-hand side, $R$. 

In this work we focus on \emph{left-connected} rewrite rules.
\begin{definition}[Strong connectivity]
An ma-hypergraph $G$ is \textbf{strongly connected} if for every input $x \in in(G)$ and output $y \in out(G)$ there exists a path from $x$ to $y$ in $G$. 

A \textbf{left-connected rewrite rule} is a span $L \xleftarrow{[i_L,o_L]} I+O \xrightarrow{[i_R,o_R]} R$ such that $I \xrightarrow{i_L} L \xleftarrow{o_L} O$ and $I \xrightarrow{i_R} R \xleftarrow{o_R} O$ are ma-cospans, $[i_L,o_L]$ is mono (we say that the rule is left-linear) and $L$ is strongly connected.

A \textbf{left-connected rewriting system} is a set of left-connected rewrite rules. 

\end{definition}

The first adaptation of DPO rewriting for string diagrams is to accommodate interfaces. This yields DPOI rewriting \cite{DBLP:journals/jacm/BonchiGKSZ22}.
%
Given two hypergraphs with interfaces, $G \xleftarrow{} J$ and $H \xleftarrow{} J$, we say that $G$ rewrites into $H$ if there exists a rewrite rule $\catspan L K R [f][g]$, a match $L \xrightarrow{m} G$ and a hypergraph with interface $C \xleftarrow{} J$, such that the squares below are pushouts and the whole diagram commutes:

\vspace{-0.3cm}
\[\begin{tikzcd}[sep = scriptsize]
	{L} && {K} && {R} \\
	{G} && {C} && {H} \\
	&& {J}
	\arrow["m"', from=1-1, to=2-1]
	\arrow[""{name=0, anchor=center, inner sep=0}, "f"', from=1-3, to=1-1]
	\arrow[""{name=1, anchor=center, inner sep=0}, "g", from=1-3, to=1-5]
	\arrow[from=1-3, to=2-3]
	\arrow[from=1-5, to=2-5]
	\arrow[from=2-3, to=2-1]
	\arrow[from=2-3, to=2-5]
	\arrow[from=3-3, to=2-1]
	\arrow[from=3-3, to=2-3]
	\arrow[from=3-3, to=2-5]
	\arrow["\lrcorner"{anchor=center, pos=0.125, rotate=90}, draw=none, from=2-1, to=0]
	\arrow["\lrcorner"{anchor=center, pos=0.125, rotate=180}, draw=none, from=2-5, to=1]
\end{tikzcd}\]
\vspace{-0.3cm}

The second adaptation of DPO rewriting is to impose convexity on matches. This is necessary to deal with the absence of a Frobenius structure \cite{DBLP:journals/mscs/BonchiGKSZ22}.

\begin{definition}[Convex matches]\label{dfn:convex_matches}
A $\Sigma$-hypergraph morphism $m : L \to G$ is a \textbf{convex match} if it is mono and its image $m(L)$ is convex, i.e. for any nodes $v, v'$ in $m(L)$ and any path $p$ from $v$ to $v'$ in $G$, every hyperedge in $p$ is also in $m(L)$.
\end{definition}

We recall the definition of boundary complement from \cite[Definition 30]{DBLP:journals/mscs/BonchiGKSZ22}.

\begin{definition}[Boundary complement]
	\label{boundary_complement_definition}
Let $I_1 \overset{i_1}{\to} G_1 \overset{o_1}{\leftarrow} O_1$ and $I_2 \overset{i_2}{\to}  G_2 \overset{o_1}{\leftarrow} O_2$ be two ma-cospans and $m : G_1 \to G_2$ a monomorphism, a pushout complement as depicted in $(\dagger)$ below is a boundary complement if $[i_1^\bot,o_1^\bot]$ is mono and there exist $i_1^\bot : I_2 \to G_1^\bot$ and $o_1^\bot : O_2 \to G_1^\bot$ making the triangle below commute and such that $O_1 + I_2 \overset{[o_1^\bot,i_2^\bot]}{\to} G_1^\bot \overset{[i_1^\bot,o_2^\bot]}{\leftarrow} I_1 + O_2$ is a ma-cospan.

\[\begin{tikzcd}[sep=small]
	{G_1} && {I_1+O_1} \\
	& {(\dagger)} \\
	{G_2} && {G_1^\bot} \\
	\\
	&& {I_2+O_2}
	\arrow["m"', from=1-1, to=3-1]
	\arrow["{[i_1,o_1]}"', from=1-3, to=1-1]
	\arrow["{[i_1^\bot,o_1^\bot]}", from=1-3, to=3-3]
	\arrow["g"{description}, from=3-3, to=3-1]
	\arrow["{[i_2,o_2]}", from=5-3, to=3-1]
	\arrow["{[i_2^\bot,o_2^\bot]}"', dashed, from=5-3, to=3-3]
\end{tikzcd}\]

\end{definition}

\begin{definition}[Convex rewriting]
Given a left-connected rewrite system $\mathcal{R}$, we say that an ma-cospan $n \xrightarrow{i_G} G \xleftarrow{o_G} m$ \textbf{rewrites convexly} into $n \xrightarrow{i_H} H \xleftarrow{o_H} m$ if there is a convex match $m' : L \to G$, a rewrite rule $L \xleftarrow{[i_L,o_L]} I+O \xrightarrow{[i_R,o_R]} R$ in $\mathcal{R}$ and a $\Sigma$-hypergraph $C$ such that the following diagram commutes, the left square is a boundary complement and the right square is a pushout:

\begin{equation*} 
\begin{tikzcd}[sep = scriptsize, ampersand replacement=\&,cramped]
	{L} \&\& {I+O} \&\& {R} \\
	{G} \&\& {C} \&\& {H} \\
	\&\& {n+m}
	\arrow["m'"', from=1-1, to=2-1]
	\arrow[""{name=0, anchor=center, inner sep=0}, "{[i_L,o_L]}"', from=1-3, to=1-1]
	\arrow[""{name=1, anchor=center, inner sep=0}, "{[i_R,o_R]}", from=1-3, to=1-5]
	\arrow[from=1-3, to=2-3]
	\arrow[from=1-5, to=2-5]
	\arrow[from=2-3, to=2-1]
	\arrow[from=2-3, to=2-5]
	\arrow["{[i_G,o_G]}", from=3-3, to=2-1]
	\arrow[from=3-3, to=2-3]
	\arrow["{[i_H,o_H]}"', from=3-3, to=2-5]
	\arrow["\lrcorner"{anchor=center, pos=0.125, rotate=90}, draw=none, from=2-1, to=0]
	\arrow["\lrcorner"{anchor=center, pos=0.125, rotate=180}, draw=none, from=2-5, to=1]
\end{tikzcd}
\end{equation*} 

We write $n \xrightarrow{i_G} G \xleftarrow{o_G} m \Rrightarrow_\mathcal{R} n \xrightarrow{i_H} H \xleftarrow{o_H} m$ and call it a \textbf{derivation}.
\end{definition}

Thanks to left-connectedness, a derivation can be uniquely determined by an ma-cospan $n \rightarrow G \leftarrow m$, a mono match $L \xrightarrow{m'} G$ and a rewrite rule $L \leftarrow I+O \rightarrow R$.
\begin{proposition}\label{prop:leftConnected_prop}
In left-connected rewrite systems, the boundary complement condition is always met.
In left-connected rewrite systems, a mono match is always convex.
\end{proposition}

\begin{proposition}\label{prop:pushoutCompl_leftConnected}
In left-connected rewrite systems, for all rewrite rules and for all mono matchings, the pushout complement $C$ always uniquely exists.
\end{proposition}

\begin{proof}
The existence follows from \cite[pp.44,45]{ehrig_fundamentals_2006} and the uniqueness of the pushout complement follows from \cite[Prop. 3.18]{DBLP:journals/jacm/BonchiGKSZ22}.

\end{proof}

\subsection{Critical Pairs} \label{sec:critical-pairs}

We finally recall the definition of critical pairs \cite{DBLP:journals/mscs/BonchiGKSZ22a}.

\begin{definition}[Critical pairs] \label{dfn:pcPair}
    Let $\mathcal{R}$ be a left-connected rewrite system, and $\catspan {L_1} {K_1} {R_1}$ and $\catspan {L_2} {K_2} {R_2}$ be its two rewrite rules. Consider two derivations with common source $n \to S \leftarrow m$:

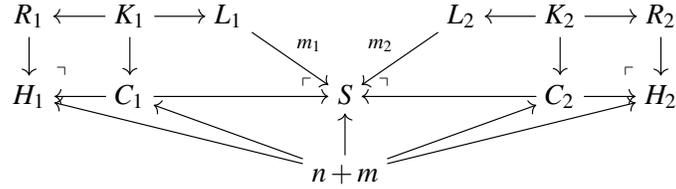
\[\begin{tikzcd}[sep = scriptsize, ampersand replacement=\&,cramped]
	{R_1} \& {K_1} \& {L_1} \&\& {L_2} \& {K_2} \& {R_2} \\
	{H_1} \& {C_1} \&\& S \&\& {C_2} \& {H_2} \\
	\&\&\& {n+m}
	\arrow[from=1-1, to=2-1]
	\arrow[from=1-2, to=1-1]
	\arrow[from=1-2, to=1-3]
	\arrow[from=1-2, to=2-2]
	\arrow["{m_1}", from=1-3, to=2-4]
	\arrow["{m_2}"', from=1-5, to=2-4]
	\arrow[from=1-6, to=1-5]
	\arrow[from=1-6, to=1-7]
	\arrow[from=1-6, to=2-6]
	\arrow[from=1-7, to=2-7]
	\arrow["\lrcorner"{anchor=center, pos=0.125, rotate=90}, draw=none, from=2-1, to=1-2]
	\arrow[from=2-2, to=2-1]
	\arrow[from=2-2, to=2-4]
	\arrow["\lrcorner"{anchor=center, pos=0.125, rotate=180}, draw=none, from=2-4, to=1-2]
	\arrow["\lrcorner"{anchor=center, pos=0.125, rotate=90}, draw=none, from=2-4, to=1-6]
	\arrow[from=2-6, to=2-4]
	\arrow[from=2-6, to=2-7]
	\arrow["\lrcorner"{anchor=center, pos=0.125, rotate=180}, draw=none, from=2-7, to=1-6]
	\arrow[from=3-4, to=2-1]
	\arrow[from=3-4, to=2-2]
	\arrow[from=3-4, to=2-4]
	\arrow[from=3-4, to=2-6]
	\arrow[from=3-4, to=2-7]
\end{tikzcd}\]

 \begin{enumerate}
  \item We say that $(n \to H_1 \leftarrow m) \Lleftarrow_\mathcal{R} (n \to S \leftarrow m) \Rrightarrow_\mathcal{R} (n \to H_2 \leftarrow m)$ is a \textbf{pre-critical pair} if $[m_1, m_2] : L_1 + L_2 \to S$ is epi.
  \item The pre-critical pair is \textbf{joinable} if there exists $W$ such that $(n \to H_1 \leftarrow m) \rightderiv (n \to W \leftarrow m) \leftderiv (n \to H_2 \leftarrow m)$, where $\rightderiv$ means a finite number (possibly zero) of rewrites.
  \item The pre-critical pair is a \textbf{parallel pair} if there exist $g_1 : L_1 \to C_2$ and $g_2 : L_2 \to C_1$ making the diagram below commute:
	%
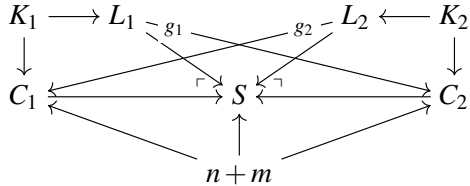
\[\begin{tikzcd}[sep = scriptsize, ampersand replacement=\&,cramped]
	{K_1} \& {L_1} \&\& {L_2} \& {K_2} \\
	{C_1} \&\& S \&\& {C_2} \\
	\&\& {n+m}
	\arrow[from=1-1, to=1-2]
	\arrow[from=1-1, to=2-1]
	\arrow[from=1-2, to=2-3]
	\arrow["{g_1}"{description, pos=0.1}, from=1-2, to=2-5]
	\arrow["{g_2}"{description, pos=0.1}, from=1-4, to=2-1]
	\arrow[from=1-4, to=2-3]
	\arrow[from=1-5, to=1-4]
	\arrow[from=1-5, to=2-5]
	\arrow[from=2-1, to=2-3]
	\arrow["\lrcorner"{anchor=center, pos=0.125, rotate=180}, draw=none, from=2-3, to=1-1]
	\arrow["\lrcorner"{anchor=center, pos=0.125, rotate=90}, draw=none, from=2-3, to=1-5]
	\arrow[from=2-5, to=2-3]
	\arrow[from=3-3, to=2-1]
	\arrow[from=3-3, to=2-3]
	\arrow[from=3-3, to=2-5]
\end{tikzcd}\]
  \item The pre-critical pair is a \emph{critical pair} if it is not parallel.
 \end{enumerate}
\end{definition}

Thanks to Prop.~\ref{prop:leftConnected_prop} and Prop.~\ref{prop:pushoutCompl_leftConnected}, for left-connected rewrite systems, the pre-critical pair $(n \to H_1 \leftarrow m) \Lleftarrow_\mathcal{R} (n \to S \leftarrow m) \Rrightarrow_\mathcal{R} (n \to H_2 \leftarrow m)$ can uniquely be determined by a cospan (dubbed \emph{cp-cospan})
$L_1 + L_2 \twoheadrightarrow S \leftarrow I+O$ such that:
(i) $L_1 + L_2 \twoheadrightarrow S$ is an epimorphism given by the coupling of two mono matches from $L_1$ and $L_2$ to $S$; and (ii) $I \rightarrow S \leftarrow O$ is a ma-cospan.

\section{A Critical Pair Enumeration Algorithm} \label{sec:algo}

Our goal is to enumerate automatically all critical pairs for a given left-connected rewrite system. To do so, we have to enumerate all relevant epimorphisms $L_i + L_j \twoheadrightarrow S$ where $L_i$ and $L_j$ are left hand sides of rewrite rules. We will glue different nodes and hyperedges of $L_i + L_j$ to enumerate these epimorphisms.

\begin{definition}
Let $L_1$ and $L_2$ be two $\Sigma$-hypergraphs.
A \textbf{gluing scheme} is given by a $\Sigma$-hypergraph $G$ and two $\Sigma$-hypergraph morphisms $g_1 : G \to L_1 + L_2$ and $g_2 : G \to L_1 + L_2$. The \textbf{gluing} is the coequalizer of $g_1$ and $g_2$.
For two nodes (or hyperedges) $x$ and $x'$ of $L_1 + L_2$, they are \emph{glued} if there exists a node (or a hyperedge) $y$ of $G$ such that $g_1(y) = x$ and $g_2(y) = x'$.
\vspace{-0.2cm}
\[\begin{tikzcd}[ampersand replacement=\&,cramped]
	G \&\& {L_1+L_2} \&\& {\mathtt{coeq}_{G}(g_1,g_2)}
	\arrow["{g_1}", curve={height=-6pt}, from=1-1, to=1-3]
	\arrow["{g_2}"', curve={height=6pt}, from=1-1, to=1-3]
	\arrow["\epsilon"', dotted, from=1-3, to=1-5]
\end{tikzcd}\]
\vspace{-0.4cm}
\end{definition}

Each gluing scheme $(g_1 \colon G \to L_1+L_2, g_2 \colon G \to L_1+L_2)$ induces a cospan:
 
\[ L_1 + L_2 \twoheadrightarrow \mathtt{coeq}_{G}(g_1,g_2) \xleftarrow{[\subseteq,\subseteq]} in(\mathtt{coeq}_{G}(g_1,g_2)) + out(\mathtt{coeq}_{G}(g_1,g_2)) \]
 We call the coequaliser $\mathtt{coeq}_{G}(g_1,g_2)$ \emph{candidate source}.
As observed in Sec.~\ref{sec:critical-pairs}, this cospan uniquely determines a pre-critical pair if it is a cp-cospan.
 Now the question is: what are necessary conditions on the gluing scheme so that the cospan becomes a cp-cospan and hence induces a (pre-)critical pair?

We first observe that a gluing scheme should not glue nodes nor hyperedges within $L_1$ and $L_2$.

\begin{proposition}
\label{prop1}

If two nodes from $L_1$ (resp. $L_2$) are glued, the gluing scheme does not yield a pre-critical pair. 
If two hyperedges from $L_1$ (resp. $L_2$) are glued, the gluing scheme does not yield a pre-critical pair. 
\end{proposition}

\begin{proof}
	If two nodes from $L_1$ are glued in $\mathtt{coeq}_{G}(g_1,g_2)$, then $\iota_1 ; \epsilon$ is not mono and is thus not a valid convex matching.
	The same proof works for hyperedges.
\end{proof}

Secondly we observe that nodes, separately from $L_1$ and $L_2$, should be glued in a specific way.

\begin{proposition} \label{prop:gluingNodes}
If a node $A$ from $L_1$ and a node $B$ from $L_2$ are glued and the gluing scheme yields a pre-critical pair, then either 
\begin{itemize}
\item $A$ and $B$ are the $k$-th source of a glued hyperedge ;
\item $A$ and $B$ are the $i$-th target of a glued hyperedge ;
\item $A$ is an output of $L_1$, $B$ is an input of $L_2$.
\item $A$ is an input of $L_1$, $B$ is an output of $L_2$;
\end{itemize}
\end{proposition}

\begin{proposition} \label{prop:gluingNodesIO}
If a node $A$ from $L_1$ and a node $B$ from $L_2$ are glued and the gluing scheme yields a pre-critical pair, then:
\begin{itemize}
\item if $A$ is an output of $L_1$, $B$ is an input of $L_2$, then no input of $L_1$ is glued to an output of $L_2$;
\item if $A$ is an input of $L_1$ and $B$ is an output of $L_2$, then no output of $L_1$ is glued to an input of $L_2$.
\end{itemize}
\end{proposition}

\begin{proof}
	Suppose that a node $A$ from $L_1$ and a node $B$ from $L_2$ are glued, the gluing scheme yields a pre-critical pair, $A$ is an output of $L_1$, and $B$ is an input of $L_2$.
	Suppose $X$ an input of $L_1$ is glued to $Y$ an output of $L_2$. There are paths $X \leadsto A$ and $B \leadsto Y$ by strong connectedness of $L_1$ and $L_2$. Then, we obtain a cycle $[X] \leadsto [A] = [B] \leadsto [Y] = [X]$ which contradicts the acyclicity property. \\
	The second point follows a similar argument.
\end{proof}

Finally, we can obtain a sufficient and necessary condition for a gluing scheme to induce a critical pair.

\begin{proposition}\label{prop:precritical_isParallel}
 A pre-critical pair
 \[ L_1 + L_2 \twoheadrightarrow \mathtt{coeq}_{G}(g_1,g_2) \xleftarrow{[\subseteq,\subseteq]} in(\mathtt{coeq}_{G}(g_1,g_2)) + out(\mathtt{coeq}_{G}(g_1,g_2))\] 
 yielded by a gluing scheme is parallel iff the following holds:
    \begin{enumerate}
        \item no hyperedges from $L_1$ and $L_2$ are glued, and
        \item if two nodes from $L_1$ and $L_2$ are glued, they are in interfaces of $L_1$ and $L_2$.
    \end{enumerate}
\end{proposition}

\begin{proof}[Proof of $\Rightarrow$]
    Suppose we have a parallel pair
\[\begin{tikzcd}[ampersand replacement=\&,cramped]
	{K_1} \& {L_1} \& {L_1 + L_2} \& {L_2} \& {K_2} \\
	\\
	{C_1} \&\& S \&\& {C_2}
	\arrow["{[i_1,o_1]}", from=1-1, to=1-2]
	\arrow[from=1-1, to=3-1]
	\arrow["{q_1}", dotted, hook, from=1-2, to=1-3]
	\arrow["{q_1;\epsilon}"{description, pos=0.3}, from=1-2, to=3-3]
	\arrow["{f_2}"{description, pos=0.8}, from=1-2, to=3-5]
	\arrow["\epsilon"{description}, two heads, from=1-3, to=3-3]
	\arrow["{q_2}"', dotted, hook', from=1-4, to=1-3]
	\arrow["{f_1}"{description, pos=0.8}, from=1-4, to=3-1]
	\arrow["{q_2;\epsilon}"{description, pos=0.3}, from=1-4, to=3-3]
	\arrow["{[i_2,o_2]}"', from=1-5, to=1-4]
	\arrow[from=1-5, to=3-5]
	\arrow[from=3-1, to=3-3]
	\arrow["\lrcorner"{anchor=center, pos=0.125, rotate=180}, draw=none, from=3-3, to=1-1]
	\arrow["\lrcorner"{anchor=center, pos=0.125, rotate=90}, draw=none, from=3-3, to=1-5]
	\arrow[from=3-5, to=3-3]
\end{tikzcd}\]
\begin{itemize}
    \item Let $v_1 \in L_1, v_2 \in L_2$ be two glued nodes (they must come from different hypergraphs by Prop.~\ref{prop1}). Since this is a parallel pair, there are mappings $f_1 : L_2 \rightarrow C_1$, $f_2 : L_1 \rightarrow C_2$ such that the diagram above commutes. Because the triangle commutes, $f_1(v_2)$ must be sent to $[v_2] = [v_1] \in S$. Moreover, $(q_1;\epsilon)(v_1) = [v_1] = [v_2] \in S$.  But $S$ is a pushout, so the identified elements of $C_1$ and $L_1$ must be present in $K_1$, so $v_1$ and $f_1(v_2)$ have a preimage in $K_1$, and $v_1 \in [i_1,o_1](K_1)$. By symmetry of the argument we have $v_2 \in [i_2,o_2](K_2)$.
    \item Suppose there are glued hyperedges. Let $e_1 \in L_1$ and $e_2 \in L_2$ be mapped to the same $e \in S$. As before, for the identified nodes we can conclude that they are in the interfaces $[i_1,o_1], [i_2,o_2]$ by diagram chasing. But $K_1, K_2$ are discrete, so such hyperedges cannot exist - there are no elements $e'_1 \in C_2$, $e'_2 \in C_1$ such that $f_2^{-1}(e'_1) = e_1$ and $f_1 ^{-1}(e'_2) = e_2$ and the diagram commutes.
\end{itemize}
\textit{Proof of $\Leftarrow$}.
    Suppose the assumptions 1 and 2 hold, let's prove that the pre-critical pair is parallel. Consider 
\[\begin{tikzcd}[ampersand replacement=\&,cramped]
	{K_1} \& {L_1} \&\& {L_2} \& {K_2} \\
	{C_1} \&\& S \&\& {C_2}
	\arrow["{[i_1,o_1]}", hook, from=1-1, to=1-2]
	\arrow["{[i_1',o_1']}"', from=1-1, to=2-1]
	\arrow["{m_1}"{description}, from=1-2, to=2-3]
	\arrow["{m_2}"{description}, from=1-4, to=2-3]
	\arrow["{[i_2,o_2]}"', hook', from=1-5, to=1-4]
	\arrow["{[i_2',o_2']}", from=1-5, to=2-5]
	\arrow["{a_1}"{description}, from=2-1, to=2-3]
	\arrow["\lrcorner"{anchor=center, pos=0.125, rotate=180}, draw=none, from=2-3, to=1-1]
	\arrow["\lrcorner"{anchor=center, pos=0.125, rotate=90}, draw=none, from=2-3, to=1-5]
	\arrow["{a_2}"{description}, from=2-5, to=2-3]
\end{tikzcd}\]
The proof is by diagram chasing.

If no nodes are glued, then no hyperedges are glued and $S \cong L_1 + L_2$ which obviously is a parallel pair.

Otherwise, consider a pair of nodes $(v_1 \in L_1$, $v_2 \in L_2)$ with the same image $v = m_1(v_1) = m_2(v_2)$ in $S$.
By assumption, $v_2$ is in the image of $[i_2,o_2] : K_2 \rightarrow L_2$. 
Let $f_2 : L_1 \rightarrow C_2$ be the morphism sending $v_1$ to $[i_2',o_2'](v_2)$ and acting as an identity on the rest of the nodes and edges (since by assumption no edges from $L_1$ and $L_2$ glued).

By construction this will make the triangle commute.

Similarly for $f_1 : L_2 \rightarrow C_1$.
\end{proof}


\begin{proposition} \label{prop:critical-pair}
 A pre-critical pair
 \[ L_1 + L_2 \twoheadrightarrow \mathtt{coeq}_{G}(g_1,g_2) \xleftarrow{[\subseteq,\subseteq]} in(\mathtt{coeq}_{G}(g_1,g_2)) + out(\mathtt{coeq}_{G}(g_1,g_2))\] 
 yielded by a gluing scheme is a critical pair iff there are hyperedges separately from $L_1$ and $L_2$ that are glued.
\end{proposition}

\begin{proof}
	This is a consequence of Prop.~\ref{prop:precritical_isParallel} and $\mathtt{coeq}_{G}(g_1,g_2)$ being an ma-hypergraph.
\end{proof}

These observations suggest the following two-fold gluing process to yield a suitable gluing scheme $(g_1 \colon G \to L_1+L_2, g_2 \colon G \to L_1+L_2)$ that induces a critical pair: (1) glue (at least one pair of) hyperedges that are separately from $L_1$ and $L_2$, and (2) glue inputs/outputs that are separately from $L_1$ and $L_2$.

To compute such a gluing scheme, we use \emph{independent edge sets} on complete bipartite graphs.
\begin{definition}
Given two sets $A$ and $B$, the \textbf{complete bipartite graph} $K_{A,B}$ is defined as follows: its vertices are $A + B$ and there is an edge between every element of $A$ and every element of $B$.
 %

An \textbf{independent edge set}\footnote{Independent edge sets are also called ``matchings'' in graph theory.} on $K_{A,B}$ is a set of edges such that no two edges share common vertices.

The following shows an example of an independent edge set.

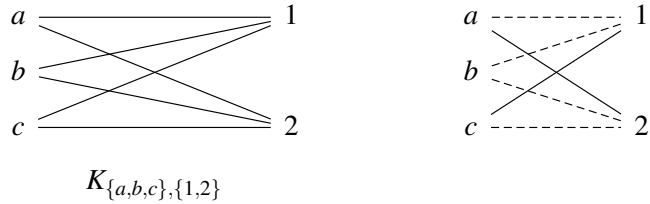
\[\begin{tikzcd}[sep = tiny]
	a &&&& 1 &&&&&&&& a &&& {} &&& 1 \\
	b &&&&&&&&&&&& b \\
	c &&&& 2 &&&&&&&& c &&& {} &&& 2 \\
	&& {K_{\{a,b,c\},\{1,2\}}}
	\arrow[no head, from=1-1, to=1-5]
	\arrow[no head, from=1-1, to=3-5]
	\arrow[dashed, no head, from=1-13, to=1-19]
	\arrow[no head, from=1-13, to=3-19]
	\arrow[no head, from=2-1, to=1-5]
	\arrow[no head, from=2-1, to=3-5]
	\arrow[dashed, no head, from=2-13, to=1-19]
	\arrow[dashed, no head, from=2-13, to=3-19]
	\arrow[no head, from=3-1, to=1-5]
	\arrow[no head, from=3-1, to=3-5]
	\arrow[no head, from=3-13, to=1-19]
	\arrow[dashed, no head, from=3-13, to=3-19]
\end{tikzcd}\]

\end{definition}

For each independent edge set on hyperedges, i.e. on $K_{\HEdges(L_1),\HEdges(L_2)}$ (or for each independent edge set on nodes, i.e. on $K_{\Nodes(L_1),\Nodes(L_2)}$), we can construct the \textbf{induced hypergraph} $\gamma$ as follows:
for each edge connecting two vertices in the independent edge set, we add the pair of hyperedges (or pair of nodes) associated with those vertices together with their induced pairs of sources and targets.
We let $p^\gamma_1 : \gamma \to L_1$ be the first projection of $\gamma$ into $L_1$ and $p^\gamma_2 : \gamma \to L_2$ be the second projection of $\gamma$ into $L_2$.
Consequently, we obtain a gluing scheme $(p^\gamma_1;\iota_1 \colon \gamma \to L_1+L_2, p^\gamma_2;\iota_2 \colon \gamma \to L_1+L_2)$.

By suitably generating independent edge sets firstly on hyperedges (i.e.\ $K_{\HEdges(L_1),\HEdges(L_2)}$), and secondly on nodes (i.e.\ $K_{\Nodes(L_1),\Nodes(L_2)}$), in particular its restriction on inputs and outputs, we can compute gluing schemes that glue hyperedges and nodes as specified by the independent edge sets and hence induces a critical pair; if there is an edge between two vertices of a bipartite graph, the two endpoints of the edge gets merged (glued).

Our algorithm uses subroutines to enumerate independent edge sets on a complete bipartite graph $K_{a,b}$. We assume that the sets $a$ and $b$ are totally ordered. We believe this is a reasonable assumption, because the set of hyperedges and the set of nodes, of a hypergraph, are typically implemented using a totally ordered data structure. The zip function turns two lists of the same length into a list of pairs.

\begin{algorithm}[H]
    \caption{$\texttt{enumerateIndependentEdgeSets}$ for enumerating independent edge sets}\label{algo1}
    \KwIn{a set $a$ and a set $b$}
    \KwOut{independent edge sets on $K_{a,b}$}
    \Return $\bigcup_{k \in \llbracket 0, \min(|a|,|b|) \rrbracket}  {\texttt{enumerateKIndependentEdgeSets}}(a,b,k)$ \;
\end{algorithm}

\begin{algorithm}[H]
    \caption{$\texttt{enumerateKIndependentEdgeSets}$ for enumerating $k$ independent edge sets}\label{algo2}
    \SetKwInOut{Input}{Input}\SetKwInOut{Output}{Output}\SetKw{Yield}{yield}
    \Input{a set $a$, a set $b$, and a number $k$}
    \Output{independent edge sets on $K_{a,b}$ with $k$ edges}
    \For{$x \subseteq a$ such that $|x| = k$}{
        \For{$y$ being a partial permutation of $y' \subseteq b$ such that $|y| = k$}{
            \Yield zip(x,y) 
        }
    }
\end{algorithm}


\begin{proposition} \label{prop:independent-edge-sets}
There are $\sum_{0 \leq k \leq min(|a|,|b|)} k! \binom{|a|}{k} \binom{|b|}{k}$ independent edge sets on $K_{a,b}$.
\end{proposition}

We can now present our critical pair enumeration algorithm (Algo.~\ref{algo3}). The implicitly defined subroutine $\texttt{InducedHypergraphs}$ computes induced hypergraphs of a given set of independent edge sets on hyperedges or nodes.

The following theorem validates Algo.~\ref{algo3}. In particular, correctness implies that the necessary conditions identified in Prop.~\ref{prop1}, Prop.~\ref{prop:gluingNodes} and Prop.~\ref{prop:gluingNodesIO} are sufficient as well.
\begin{theorem} \label{thm:main}
	Algo.~\ref{algo3} is correct and exhaustive. That is,
    \begin{description}
        \item[Correctness] \label{item:correctness} Each result $L_i + L_j \overset{\epsilon_{ij\gamma'}}{\twoheadrightarrow} S_{ij\gamma'} \xleftarrow{[\subseteq,\subseteq]} I'+O'$ of Algo.~\ref{algo3} is a critical pair.
        \item[Exhaustiveness] \label{item:exhaustiveness} Any critical pair of the form $L_i + L_j \twoheadrightarrow X \xleftarrow{[\subseteq,\subseteq]} in(X)+out(X)$ can be yielded by Algo.~\ref{algo3}.
    \end{description}
\end{theorem}
\begin{proof}[Proof of correctness.]
	A coequalizer is an epimorphism, thus $\epsilon_{ij\gamma}$ and $\epsilon_{ij\gamma'}$ is epi. Their composition is therefore epi. $I' \xrightarrow{\subseteq} S_{ij\gamma'} \xleftarrow{\subseteq} O' $ is a ma-cospan as required by the pre-critical pair with interface definition because of the if statement in line 15. Moreover, the matchings $\iota_1;\epsilon_{ij\gamma};\epsilon_{ij\gamma'}$ and $\iota_2;\epsilon_{ij\gamma};\epsilon_{ij\gamma'}$ are mono, because the gluing schemes $\gamma$ and $\gamma'$ do not glue nodes and hyperedges of the same hypergraph. The gluing scheme $\gamma$ also glues at least a pair of hyperedges (line 4). Therefore, by Prop.~\ref{prop:critical-pair}, each result yielded is a critical pair.

\textit{Proof of exhaustiveness.}
Let $L_i + L_j \overset{\epsilon}{\twoheadrightarrow} X \xleftarrow{[\subseteq,\subseteq]} in(X)+out(X)$ be a pre-critical pair with interface.

An epimorphism of hypergraphs is surjective on nodes and on hyperedges because the category of hypergraphs is a presheaf category.

Thus, each node and each hyperedge of $X$ has a non-empty preimage set by $\epsilon$.
Moreover, each preimage set by $\epsilon$ contains: (i) at most two elements, and (ii) if there are two elements, they come separately from $L_1$ and $L_2$. This is because $\iota_i;\epsilon$ and $\iota_j;\epsilon$ are mono.

We construct a hypergraph $\gamma$ whose nodes are given by preimage sets $\epsilon^{-1}(v)$ with size $2$ for $v \in \Nodes(X)$, and hyperedges are given by preimage sets $\epsilon^{-1}(e)$ with size $2$ for $e \in \HEdges(X)$. It comes with two hypergraph homomorphisms $F_1 \colon \gamma \to L_1+L_2$ and $F_2 \colon \gamma \to L_1+L_2$, such that $F_1$ maps a preimage set to its element from $L_1$ and $F_2$ maps a preimage set to its element from $L_2$. We obtain a gluing scheme $(F_1,F_2)$, and $X$ is the gluing of $(F_1,F_2)$.

Because this gluing scheme induces a critical pair, it satisfies the necessary conditions of the propositions in Sec.~\ref{sec:algo}. Namely:
\begin{enumerate}
 \item If it glues edges, they are separately from $L_1$ and $L_2$.
 \item It glues at least a pair of hyperedges separately from $L_1$ and $L_2$.
 \item If it glues nodes, they are either a source/target of glued hyperedges, or input/output separately from $L_1$ and $L_2$.
\end{enumerate}
These conditions are realised by Algo.~\ref{algo3}, respectively by lines 2 \& 10, line 4, and line 10.

 We can therefore conclude that the merging of hyperedges and nodes specified by the gluing scheme $(F_1,F_2)$ is implemented by Algo.~\ref{algo3}.

\end{proof}

We implement Algo.~\ref{algo3} in Haskell\footnote{Available online at \url{https://github.com/GuiSab/hypergraphrewriting}}, and test it using the example of non-commutative bimonoids \cite[Sec.~6.1]{DBLP:journals/mscs/BonchiGKSZ22a}. While there are 22 critical pairs, the implementation outputs 58 critical pairs. This is due to duplication caused by isomorphic gluing schemes $\gamma, \gamma'$. Our implementation currently does not check for isomorphisms of hypergraphs.


\begin{algorithm}[H]
    \caption{An algorithm for enumerating all critical pairs}\label{algo3}
    \SetKwInOut{Input}{Input}\SetKwInOut{Output}{Output}\SetKw{Yield}{yield}
    \Input{rewrite rules $\rho = \{ L_i \overset{f_i}{\leftarrow} K_i \overset{g_i}{\rightarrow} R_i \}_{i \in I}$}
    \Output{epimorphisms with interface $\{ \{ L_i + L_j \overset{\epsilon}{\twoheadrightarrow} S_{ij\gamma} \leftarrow I+O \}_{\gamma \in I_{ij}} \}_{(i,j) \in I^2}$}
    \For{$(i,j) \in I^2$}{
        \For{$\gamma \in \texttt{InducedHypergraphs}(\bigcup\limits_{l \in \Sigma}^{} \texttt{enumerateIndependentEdgeSets}($\\
        $\{e \mid e \in Hyperedges(L_i), label(e) = l\},\{e \mid e \in Hyperedges(L_j), label(e) = l\}))$}{
            \If{$\gamma$ has at least a hyperedge}{
                $(S_{ij\gamma}, \epsilon_{ij\gamma}) = \mathtt{coeq}_{\gamma}(p^\gamma_1;\iota_1, p^\gamma_2;\iota_2)$\;
                \tcc{the coequalizer (
                    \begin{tikzcd}[sep = scriptsize, ampersand replacement=\&,cramped]
                        \gamma \&\& {L_i+L_j} \&\& {S_{ij\gamma}}
                        \arrow["{p^{\gamma}_1;\iota_1}", curve={height=-6pt}, from=1-1, to=1-3]
                        \arrow["{p^{\gamma}_2;\iota_2}"', curve={height=6pt}, from=1-1, to=1-3]
                        \arrow["{\epsilon_{ij\gamma}}", dotted, from=1-3, to=1-5]
                    \end{tikzcd} in $\mathbf{Hyp}_\Sigma$)}
                $I_1 = in(S_{ij\gamma}) \cap \epsilon_{ij\gamma}(\iota_1(in(L_i)))$\;
                $I_2 = in(S_{ij\gamma}) \cap \epsilon_{ij\gamma}(\iota_2(in(L_j)))$\;
                $O_1 = out(S_{ij\gamma}) \cap \epsilon_{ij\gamma}(\iota_1(out(L_i)))$\;
                $O_2 = out(S_{ij\gamma}) \cap \epsilon_{ij\gamma}(\iota_2(out(L_j)))$\;
                \For{$\gamma' \in \texttt{InducedHypergraphs}(\texttt{enumerateIndependentEdgeSets}(I_1,O_2) +$\\
                     $\texttt{enumerateIndependentEdgeSets}(I_2,O_1))$}{
                    $(S_{ij\gamma'}, \epsilon_{ij\gamma'}) = \mathtt{coeq}_{\gamma'}(p^{\gamma'}_1, p^{\gamma'}_2)$\;
                    \tcc{the coequalizer (
\begin{tikzcd}[ampersand replacement=\&,cramped]
	{\gamma'} \&\& {S_{ij\gamma}} \&\& {S_{ij\gamma'}}
	\arrow["{p^{\gamma'}_1}", curve={height=-6pt}, from=1-1, to=1-3]
	\arrow["{p^{\gamma'}_2}"', curve={height=6pt}, from=1-1, to=1-3]
	\arrow["{\epsilon_{ij\gamma'}}", dotted, from=1-3, to=1-5]
\end{tikzcd} in $\mathbf{Hyp}_\Sigma$)}
                    $I' = in(S_{ij\gamma'})$\;
                    $O' = out(S_{ij\gamma'})$\;
                    \If{$I' \xrightarrow{\subseteq} S_{ij\gamma'} \xleftarrow{\subseteq} O' $ is a ma-cospan}{
                
                        \Yield{$L_i + L_j \overset{\epsilon_{ij\gamma};\epsilon_{ij\gamma'}}{\twoheadrightarrow} S_{ij\gamma'} \xleftarrow{[\subseteq,\subseteq]} I'+O'$}\;
                    }
                }
            }
        }
    }
\end{algorithm}


\begin{example}
    
We compute critical pairs associated to the following pair of rules that is taken from the example of non-commutative bimonoids \cite[Sec.~6.1]{DBLP:journals/mscs/BonchiGKSZ22a}.

\[\begin{tikzcd}[cramped,sep=tiny, scale cd = 0.7]
	&& {L_1} &&& {} && {} && {K_1} && {} && {} &&& {R_1} \\
	\\
	0 &&&&&&&& 0 &&&&&&&& 0 \\
	& {\boxed{\mu}_1} & 4 &&& {} && {} & 1 && 3 & {} && {} & 1 &&& {\boxed{\mu}_4} & 3 \\
	1 &&& {\boxed{\mu}_2} & 3 &&&& 2 &&&&&&& {\boxed{\mu}_3} & 4 \\
	&& 2 &&&&&&&&&&&& 2
	\arrow[from=1-8, to=1-6]
	\arrow[from=1-12, to=1-14]
	\arrow[no head, from=3-1, to=4-2]
	\arrow[no head, from=3-17, to=4-18]
	\arrow[from=4-2, to=4-3]
	\arrow[no head, from=4-3, to=5-4]
	\arrow[from=4-8, to=4-6]
	\arrow[from=4-12, to=4-14]
	\arrow[no head, from=4-15, to=5-16]
	\arrow[from=4-18, to=4-19]
	\arrow[no head, from=5-1, to=4-2]
	\arrow[from=5-4, to=5-5]
	\arrow[from=5-16, to=5-17]
	\arrow[from=5-17, to=4-18]
	\arrow[no head, from=6-3, to=5-4]
	\arrow[no head, from=6-15, to=5-16]
\end{tikzcd}\]

\[\begin{tikzcd}[cramped,sep=tiny, scale cd = 0.7]
	& {L_2} &&& {} && {} && {K_2} && {} && {} &&& {R_2} \\
	& 5 \\
	&& {\boxed{\mu}_1} & 6 &&&& 5 && 6 &&&&&& {5=6} \\
	{\boxed{\eta}_1} & 7
	\arrow[from=1-7, to=1-5]
	\arrow[from=1-11, to=1-13]
	\arrow[no head, from=2-2, to=3-3]
	\arrow[from=3-3, to=3-4]
	\arrow[from=4-1, to=4-2]
	\arrow[no head, from=4-2, to=3-3]
\end{tikzcd}\]

We first enumerate the independent edge sets associated to the labels $\mu$: there are three independent edge sets, namely $\{\},\{(\boxed{\mu}_1,\boxed{\mu}_1)\}$ and $\{(\boxed{\mu}_2,\boxed{\mu}_1)\}$.
There is only one independent edge set for the label $\eta$, namely the empty one.
We thus have 2 gluing schemes for the hyperedges which are not empty:

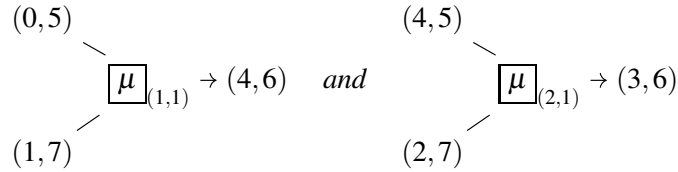
\[\begin{tikzcd}[ampersand replacement=\&,cramped, sep = tiny]
	{(0,5)} \&\&\&\& {(4,5)} \\
	\& {\boxed{\mu}_{(1,1)}} \& {(4,6)} \& {\text{and}} \&\& {\boxed{\mu}_{(2,1)}} \& {(3,6)} \\
	{(1,7)} \&\&\&\& {(2,7)}
	\arrow[no head, from=1-1, to=2-2]
	\arrow[no head, from=1-5, to=2-6]
	\arrow[from=2-2, to=2-3]
	\arrow[from=2-6, to=2-7]
	\arrow[no head, from=3-1, to=2-2]
	\arrow[no head, from=3-5, to=2-6]
\end{tikzcd}\]

The gluings associated to the gluing schemes are the following:

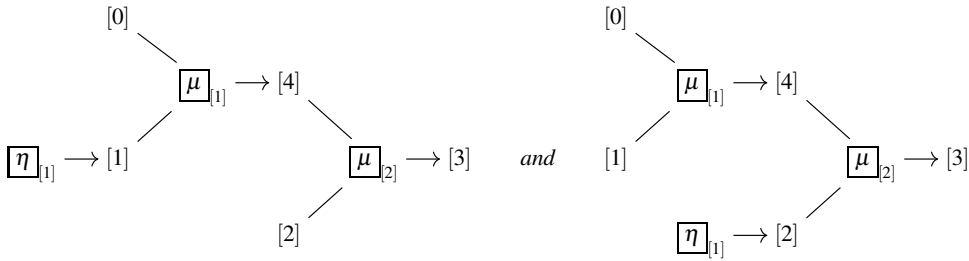
\[\begin{tikzcd}[ampersand replacement=\&,cramped, sep = small, scale cd = 0.8]
	\& {[0]} \&\&\&\&\&\& {[0]} \\
	\&\& {\boxed{\mu}_{[1]}} \& {[4]} \&\&\&\&\& {\boxed{\mu}_{[1]}} \& {[4]} \\
	{\boxed{\eta}_{[1]}} \& {[1]} \&\&\& {\boxed{\mu}_{[2]}} \& {[3]} \& {\text{and}} \& {[1]} \&\&\& {\boxed{\mu}_{[2]}} \& {[3]} \\
	\&\&\& {[2]} \&\&\&\&\& {\boxed{\eta}_{[1]}} \& {[2]}
	\arrow[no head, from=1-2, to=2-3]
	\arrow[no head, from=1-8, to=2-9]
	\arrow[from=2-3, to=2-4]
	\arrow[no head, from=2-4, to=3-5]
	\arrow[from=2-9, to=2-10]
	\arrow[no head, from=2-10, to=3-11]
	\arrow[from=3-1, to=3-2]
	\arrow[no head, from=3-2, to=2-3]
	\arrow[from=3-5, to=3-6]
	\arrow[no head, from=3-8, to=2-9]
	\arrow[from=3-11, to=3-12]
	\arrow[no head, from=4-4, to=3-5]
	\arrow[from=4-9, to=4-10]
	\arrow[no head, from=4-10, to=3-11]
\end{tikzcd}\]

For the first gluing, we compute $I_1 = \{[0],[2]\}$, $I_2 = \{[0]\}$, $O_1 = \{[3]\}$, $O_2 = \{\}$. The only independent edge sets on the nodes are $\{\}$ and $\{([0],[3])\}$. The gluing associated to  $\{([0],[3])\}$ is not acyclic:

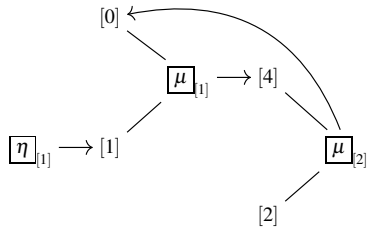
\[\begin{tikzcd}[ampersand replacement=\&,cramped, sep = small, scale cd = 0.7]
	\& {[0]} \\
	\&\& {\boxed{\mu}_{[1]}} \& {[4]} \\
	{\boxed{\eta}_{[1]}} \& {[1]} \&\&\& {\boxed{\mu}_{[2]}} \\
	\&\&\& {[2]}
	\arrow[no head, from=1-2, to=2-3]
	\arrow[from=2-3, to=2-4]
	\arrow[no head, from=2-4, to=3-5]
	\arrow[from=3-1, to=3-2]
	\arrow[no head, from=3-2, to=2-3]
	\arrow[curve={height=30pt}, from=3-5, to=1-2]
	\arrow[no head, from=4-4, to=3-5]
\end{tikzcd}\]

Therefore, we only yield the critical pair given by the first gluing.


For the second gluing, we compute $I_1 = \{[0],[1]\}$, $I_2 = \{\}$, $O_1 = \{[3]\}$, $O_2 = \{3\}$. The independent edge sets on the nodes are $\{\}$, $\{([0],[3])\}$ and $\{([1],[3])\}$. The gluing associated to  $\{([0],[3])\}$ and $\{([1],[3])\}$ are not acyclic:

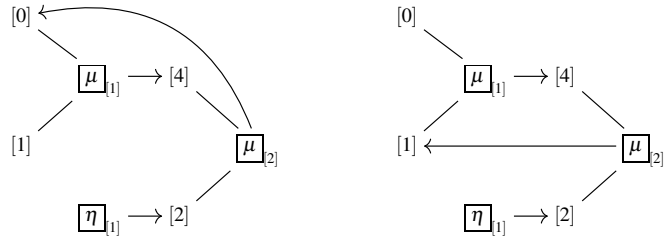
\[\begin{tikzcd}[ampersand replacement=\&,cramped, sep = small, scale cd = 0.7]
	{[0]} \&\&\&\&\&\& {[0]} \\
	\& {\boxed{\mu}_{[1]}} \& {[4]} \&\&\&\&\& {\boxed{\mu}_{[1]}} \& {[4]} \\
	{[1]} \&\&\& {\boxed{\mu}_{[2]}} \&\&\& {[1]} \&\&\& {\boxed{\mu}_{[2]}} \\
	\& {\boxed{\eta}_{[1]}} \& {[2]} \&\&\&\&\& {\boxed{\eta}_{[1]}} \& {[2]}
	\arrow[no head, from=1-1, to=2-2]
	\arrow[no head, from=1-7, to=2-8]
	\arrow[from=2-2, to=2-3]
	\arrow[no head, from=2-3, to=3-4]
	\arrow[from=2-8, to=2-9]
	\arrow[no head, from=2-9, to=3-10]
	\arrow[no head, from=3-1, to=2-2]
	\arrow[curve={height=30pt}, from=3-4, to=1-1]
	\arrow[no head, from=3-7, to=2-8]
	\arrow[from=3-10, to=3-7]
	\arrow[from=4-2, to=4-3]
	\arrow[no head, from=4-3, to=3-4]
	\arrow[from=4-8, to=4-9]
	\arrow[no head, from=4-9, to=3-10]
\end{tikzcd}\]

Therefore, we only yield the critical pair given by the second gluing: 

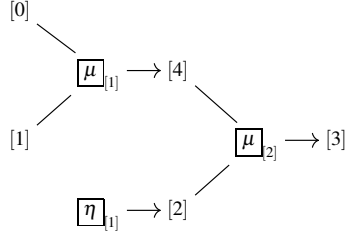
\[\begin{tikzcd}[ampersand replacement=\&,cramped, sep = small, scale cd = 0.7]
	{[0]} \\
	\& {\boxed{\mu}_{[1]}} \& {[4]} \\
	{[1]} \&\&\& {\boxed{\mu}_{[2]}} \& {[3]} \\
	\& {\boxed{\eta}_{[1]}} \& {[2]}
	\arrow[no head, from=1-1, to=2-2]
	\arrow[from=2-2, to=2-3]
	\arrow[no head, from=2-3, to=3-4]
	\arrow[no head, from=3-1, to=2-2]
	\arrow[from=3-4, to=3-5]
	\arrow[from=4-2, to=4-3]
	\arrow[no head, from=4-3, to=3-4]
\end{tikzcd}\]
\end{example}

\section{An Optimisation} \label{sec:opt}

Algo.~\ref{algo3} implements the two-fold gluing process, firstly gluing hyperedges and secondly gluing inputs/outputs.
We can in fact prove that the second step is redundant, for the purpose of critical pair analysis (and local-confluence check).

Let $S_{ij\gamma}$ be a gluing of hyperedges of Algo.~\ref{algo3} and $S_{ij\gamma'}$ a gluing of nodes on $S_{ij\gamma}$.

\begin{proposition} \label{prop:cp-optimisation}
If $S_{ij\gamma'}$ yields a critical pair, then $S_{ij\gamma}$ yields a critical pair as well.
\end{proposition}

\begin{proof}
If $S_{ij\gamma'}$ is monogamous acyclic, then the hypergraph $S_{ij\gamma}$ in which no nodes were glued cannot be cyclic; moreover, it will respect the monogamy condition.
 %
\end{proof}

We now suppose that $S_{ij\gamma'}$ yields a critical pair (and thus $S_{ij\gamma}$ yields a critical pair as well).

\begin{proposition}\label{prop:convexMatch_optimisation}
Any convex match in $I \to S_{ij\gamma} \leftarrow O$ induces a convex match in $I' \to S_{ij\gamma'} \leftarrow O'$.
\end{proposition}

\begin{proof}
Let $m : L \to S_{ij\gamma}$ be a convex match. We will prove that $m;\epsilon_{ij\gamma'} : L \to S_{ij\gamma'}$ is a convex match as well.


Let $e_1$ and $e_2$ be two hyperedges of $L$ such that $(m;\epsilon_{ij\gamma'})(e_1) = (m;\epsilon_{ij\gamma'})(e_2)$, meaning that $\epsilon_{ij\gamma'}(m(e_1)) = \epsilon_{ij\gamma'}(m(e_2))$.
We must have $ m(e_1) = m(e_2)$ because $\epsilon_{ij\gamma'}$ only glues nodes, it is thus mono on hyperedges. We then deduce $e_1 = e_2$ because $m$ is mono. $m;\epsilon_{ij\gamma'}$ is therefore mono on hyperedges.

Let $v_1$ and $v_2$ be two nodes of $L$ such that $(m;\epsilon_{ij\gamma'})(v_1) = (m;\epsilon_{ij\gamma'})(v_2)$, meaning that $\epsilon_{ij\gamma'}(m(v_1)) = \epsilon_{ij\gamma'}(m(v_2))$. 

Suppose $m(v_1) \neq m(v_2)$. $\epsilon_{ij\gamma'}$ would glue $m(v_1)$ with $m(v_2)$. By construction of $\epsilon_{ij\gamma'}$ we either have $m(v_1)$ an input of $S_{ij\gamma}$ and  $m(v_2)$ an output of $S_{ij\gamma}$ or $m(v_1)$ an output of $S_{ij\gamma}$ and  $m(v_2)$ an input of $S_{ij\gamma}$. Let's suppose WLOG the first case. As $m$ is a convex match, there is a path $[m(e_1),\cdots,m(e_n)]$ from $m(v_1)$ to $m(v_2)$ in the image of $m$. By applying $\epsilon_{ij\gamma'}$ to the path, we get $[\epsilon_{ij\gamma'}(m(e_1)),\cdots,\epsilon_{ij\gamma'}(m(e_2))]$ a path from $\epsilon_{ij\gamma'}(m(v_1))$ to $\epsilon_{ij\gamma'}(m(v_2)) =\epsilon_{ij\gamma'}(m(v_1)) $. $S_{ij\gamma'}$ would not be acyclic which contradicts the hypothesis that $S_{ij\gamma'}$ yields a valid critical pair. Therefore we proved $m(v_1) = m(v_2)$. We then deduce $v_1 = v_2$ because $m$ is mono. $m;\epsilon_{ij\gamma'}$ is therefore mono on nodes.

$m;\epsilon_{ij\gamma'}$ is mono, we can therefore deduce that $m;\epsilon_{ij\gamma'} : L \to S_{ij\gamma'}$ is a convex match.
\end{proof}


\begin{corollary}
Any rewriting sequence on $I \to S_{ij\gamma} \leftarrow O$ induces a rewriting sequence on $I' \to S_{ij\gamma'} \leftarrow O'$.
\end{corollary}

\begin{corollary}\label{cor:joinability_min_condition}
If $I \to S_{ij\gamma} \leftarrow O$ yields a joinable critical pair, then so does $I' \to S_{ij\gamma'} \leftarrow O'$.
\end{corollary}

By Cor.~\ref{cor:joinability_min_condition}, it suffices to enumerate the critical pairs where only hyperedges are glued, to determine if a left-connected rewrite system is locally confluent or not. Algo.~\ref{algo4} enumerates a sufficient subset of critical pairs necessary to determine local confluence.


\begin{algorithm}[H]
    \caption{An algorithm for enumerating sufficient critical pairs}\label{algo4}
    \SetKwInOut{Input}{Input}\SetKwInOut{Output}{Output}\SetKw{Yield}{yield}
    \Input{rewrite rules $\rho = \{ L_i \overset{f_i}{\leftarrow} K_i \overset{g_i}{\rightarrow} R_i \}_{i \in I}$}
    \Output{epimorphisms with interface $\{ \{ L_i + L_j \overset{\epsilon}{\twoheadrightarrow} S_{ij\gamma} \leftarrow I+O \}_{\gamma \in I_{ij}} \}_{(i,j) \in I^2}$}
    \For{$(i,j) \in I^2$}{
        \For{$\gamma \in \texttt{InducedHypergraphs}(\prod\limits_{l \in \Sigma}^{} \texttt{enumerateIndependentEdgeSets}($\\
        $\{e \mid e \in Hyperedges(L_i), label(e) = l\},\{e \mid e \in Hyperedges(L_j), label(e) = l\})$}{
            \If{$\gamma$ has at least a hyperedge}{
                $(S_{ij\gamma}, \epsilon_{ij\gamma}) = \mathtt{coeq}_{\gamma}(p^\gamma_1;\iota_1, p^\gamma_2;\iota_2)$\;
                \tcc{the coequaliser of (
                    \begin{tikzcd}[sep = scriptsize, ampersand replacement=\&,cramped]
                        \gamma \&\& {L_1+L_2} \&\& {S_{ij\gamma}}
                        \arrow["{p^{\gamma}_1;\iota_1}", curve={height=-6pt}, from=1-1, to=1-3]
                        \arrow["{p^{\gamma}_2;\iota_2}"', curve={height=6pt}, from=1-1, to=1-3]
                        \arrow["{\epsilon_{ij\gamma}}", dotted, from=1-3, to=1-5]
                    \end{tikzcd} in $\mathbf{Hyp}_\Sigma$)}
                $I = in(S_{ij\gamma})$\;
                $O = out(S_{ij\gamma})$\;
                \If{$I \xrightarrow{\subseteq} S_{ij\gamma} \xleftarrow{\subseteq} O $ is a ma-cospan}{
            
                    \Yield{$L_i + L_j \overset{\epsilon_{ij\gamma}}{\twoheadrightarrow} S_{ij\gamma} \xleftarrow{[\subseteq,\subseteq]} I+O$}\;
                    }
             
            }
        }
     }
\end{algorithm}

\section{Conclusion and Future Work} \label{sec:conclusion}

In this paper, we presented an algorithm that enumerates all critical pairs of a given left-connected convex DPOI rewrite systems. We proved its correctness and exhaustiveness.
The algorithm is centered around the two-fold process of gluing, firstly merging hyperedges and secondly merging inputs/outputs.
We further presented an optimisation of the algorithm that only merges hyperedges.

We are interested in complexity analysis of our algorithms (Algo.~\ref{algo3} and Algo.~\ref{algo4}), in particular evaluation of effectiveness of the optimisation.
Left-connectivity is crucial in our development, but it may be possible to extend Algo.~\ref{algo3} to non-left-connected rewrite systems, using \emph{formal path extensions} \cite{DBLP:journals/mscs/BonchiGKSZ22a}.
Another future direction is extension to string diagrams in monoidal closed categories, for which DPOI rewriting has been studied \cite{DBLP:conf/fscd/Alvarez-Picallo22} but critical pair analysis has yet been established.

\renewcommand{\emph}[1]{\textit{#1}}
\bibliographystyle{eptcs}
\bibliography{ref}

\clearpage

\appendix

\end{document}

%% file: preamble.tex
\usepackage{ulem}

\usepackage{quiver}
\usepackage{stmaryrd}
\usepackage{enumitem}
\usepackage{tikz}
\usetikzlibrary {cd,arrows.meta}
\usepackage{tikz-cd}
\usepackage{adjustbox}
\usetikzlibrary{positioning}
\usetikzlibrary{graphs}
\usetikzlibrary{backgrounds}
\usetikzlibrary{arrows.meta}
\usepackage{float} 

\newcommand \hide [1] {\bgroup}
\tikzset{
	mgvertex/.style={
		circle,
		fill,
		inner sep=0pt,
		minimum size=5pt,
		label=#1:{\tikzgraphnodetext},
		execute at begin node=\hide,
		graphs/as={}
	},
	mgvertex/.default={above},
	mg/.style={
		graphs/every graph/.style = {grow right sep = 15mm},
		every edge quotes/.style = {draw,fill=white,sloped,midway,->},
		baseline = (current bounding box.center)
	},
	k1/.pic = {
		\graph {
			0 [mgvertex=left] -!- 1 [mgvertex=right] ;
			2 [mgvertex=left] -!- 3 [mgvertex=right] ;
		} ;
	},
	l1/.pic = {
		\graph {
			0 [mgvertex=left] -> ["a"] 1 [mgvertex=right] ;
			2 [mgvertex=left] -> ["b"] 3 [mgvertex=right] ;
		} ;
	},
	l2/.pic = {
		\graph {
			4 [mgvertex=left] -> ["a"] 5 [mgvertex=right];
		} ;
	},
	k2/.pic = {
		\graph {
			4 [mgvertex=left]-!- 5 [mgvertex=right];
		} ;
	}
}
\usepackage{amssymb}
\usepackage{mathtools}
\usepackage{hyperref}
\usepackage{tikz-cd}
\usepackage[ruled,linesnumbered]{algorithm2e}

\newcommand{\leftderiv}{\stackrel{*}{\Lleftarrow}_\mathcal{R}}
\newcommand{\rightderiv}{\stackrel{*}{\Rrightarrow}_\mathcal{R}}

\NewDocumentCommand \catspan   {mmmO{}O{}} {{#1 \xlongleftarrow{#4} #2 \xlongrightarrow{#5} #3}}
\NewDocumentCommand \catcospan {mmmO{}O{}} {{#1 \xlongrightarrow{#4} #2 \xlongleftarrow{#5} #3}}

\tikzcdset{scale cd/.style={every label/.append style={scale=#1},
    cells={nodes={scale=#1}}}}
    
\tikzcdset{scale label/.style={every label/.append style={scale=#1}}}
    
\tikzcdset{scale node/.style={cells={nodes={scale=#1}}}}

\setlist{topsep=0pt, partopsep=0pt, itemsep=2pt, parsep=0pt}
\newcommand{\Hyp}{\ensuremath{\mathbf{Hyp}} }

\usepackage{amsthm}
\theoremstyle{definition}
\newtheorem{definition}{Definition}[section]

\theoremstyle{plain}
\newtheorem{proposition}[definition]{Proposition}

\newtheorem{theorem}[definition]{Theorem}
\newtheorem{corollary}[definition]{Corollary}

\newtheorem{example}[definition]{Example}

\newtheorem*{notation*}{Notation}
\newcommand{\thistheoremname}{}
\newtheorem*{generic*}{\thistheoremname}
\newenvironment{freenamed*}[1]
  {\renewcommand{\thistheoremname}{#1}%
   \begin{generic*}}
  {\end{generic*}}
\theoremstyle{remark}

\renewcommand{\emph}[1]{\textbf{#1}}

\DeclareMathOperator{\Nodes}{\mathit{Nodes}}
\DeclareMathOperator{\HEdges}{\mathit{HEdges}}